\documentclass[12pt]{amsart}
\usepackage{amsfonts}
\usepackage{latexsym}
\usepackage{amssymb}
\usepackage{amsmath}

\renewcommand{\H}{\mathbb{H}}

\newcommand{\N}{\mathbb{N}}

\newcommand{\R}{\mathbb{R}}

\newcommand{\cH}{\mathcal{H}}

\newcommand{\cL}{\mathcal{L}}

\newcommand{\diam}{\mbox{diam}}

\newcommand{\ep}{\varepsilon}

\newcommand{\ph}{\varphi}

\newcommand{\sm}{\setminus}

\newcommand{\ra}{\rightarrow}
\newcommand{\lra}{\longrightarrow}

\newtheorem{The}{Theorem}
\newtheorem{Lem}{Lemma}
\newtheorem{Rem}{Remark}
\newtheorem{Def}{Definition}

\newtheorem{Exa}{Example}

\begin{document}

\title
[An area formula in metric spaces]
{{\bf An area formula in metric spaces}}
\author{Valentino Magnani}
\address{Valentino Magnani, Dipartimento di Matematica \\
L.go Pontecorvo 5, I-56127, Pisa}
\email{magnani@dm.unipi.it}

\begin{abstract}
We present an area formula for continuous mappings between metric spaces,
under minimal regularity assumptions. In particular, we do not require 
any notion of differentiability. This is a consequence of a measure 
theoretic notion of Jacobian, defined as the density of a suitable 
``pull-back measure''.
\end{abstract}

\maketitle

\medskip
\bigskip

Let $(X,d,\mu)$ and $(Y,\rho,\nu)$ be two metric measure spaces,
where $\mu$ is a Borel regular measure on $X$ and $\nu$ is a Borel measure on $Y$.
The terminology ``measure'' refers to a countably subadditive 
nonnegative set function, see 2.1.2 of \cite{Fed}. 
We also assume that $\mu$ is finite on bounded sets
and that there exists a {\em $\mu$ Vitali relation} $V$, 2.8.16 of \cite{Fed}.

The first point is the notion of ``pull-back measure''
with respect to a continuous mapping. To do this, we 
need the following important result, proved in 2.2.13 of \cite{Fed}. 
{\em Let $X$ be a complete and separable metric space and let 
$g:X\lra Y$ be continuous. Then for every Borel set $B\subset X$,
we have that $g(B)$ is $\nu$-measurable.}

Throughout, the above assuptions will constitute our underlying assumptions.
\begin{Def}[Pull-back measure]\label{pullb}\rm
Let $(X,d)$ be complete and separable, let $E\subset X$ be closed and 
let $f:E\lra Y$ be continuous. For each $S\subset E$, we set
$\zeta(S)=\nu\big(f(S)\big)$.
We denote by $f^*\nu$ the measure arising from the Caratheodory's construction
applied with $\zeta$ defined on the family of Borel sets, according to 2.10.1 of \cite{Fed}. 
We say that $f^*\nu$ is the {\em pull-back measure} of $\nu$ with respect to $f$.
The measure $f^*\nu$ is automatically extended to the whole of $X$, by setting
$f^*\nu(A)=f^*\nu(A\cap E)$ for any $A\subset X$.
\end{Def}
In the sequel, $E$ will stand for any closed subset of $X$.
Notice that $f^*\nu$ is a Borel regular measure on $E$, as it follows by the Carath\'eodory construction.

%
%
Recall that the {\em multiplicity function} of $f:E\lra Y$ relative to $A$ is defined 
as $N(f,A,y)=\#\big(A\cap f^{-1}(y)\big)$ for all $y\in Y$.
For any Borel set $T\subset E$, Theorem 2.10.10 of \cite{Fed} gives us the formula
\begin{equation}\label{Borelf}
f^*\nu(T)=\int_YN(f,T,y)\,d\nu(y)\,.
\end{equation}
\footnoterule{
The author has been supported by ''ERC ADG Grant GeMeTneES'' 
}
\begin{Rem}\rm
It is important to notice that when $f^*\nu$ is absolutely continuous with respect to $\mu$
and finite on bounded sets, standard arguments show that 
\begin{equation}\label{extends}
f^*\nu(A)=\int_YN(f,A,y)\,d\nu(y)
\end{equation}
for any $\mu$-measurable set $A\subset E$, extending \eqref{Borelf} to $\mu$-measurable sets.
\end{Rem}
We are now lead to two notions of metric Jacobian.
%
%
%
%
%
\begin{Def}[Metric Jacobian]\label{jacmetI}{\rm
Let $f:E\lra Y$ be continuous and let $x\in E$.
Then we introduce two {\em metric Jacobians} of $f$ at $x$ as follows 
\begin{equation}\label{jac}
J_f(x)=(V)\limsup_{S\to x}\frac{\nu\big(f(S\cap E)\big)}{\mu(S)}\quad\mbox{and}
\quad Jf(x)=(V)\limsup_{S\to x}\frac{f^*\nu(S)}{\mu(S)}\,.
\end{equation}
}\end{Def}
\noindent
From 2.8.16 of \cite{Fed}, for each $\overline{\R}$-valued function $\ph$
defined on a subset of $V$, we have
\[
(V)\limsup_{S\to x}\ph(S)=\lim_{\ep\to0^+} \sup\big\{\ph(S): 
(x,S)\in V,\,S\in\mbox{dmn}(\ph),\, \diam(S)<\ep \big\}\,,
\]
where $\mbox{dmn}(\ph)$ denotes the domain of $\ph$. It is understood that
$(V) \lim$ and $(V) \liminf$ are introduced in analogous way.

In the sequel, we will present in two distinct theorems the metric area formula 
under slightly different assumptions, that depend on the notion of metric Jacobian we use.
This essentially provides an {\em axiomatic approach} to the area formula 
in a metric setting, without appealing to any notion of differentiability. 
%
%
%
%
%
%
\begin{The}[Area formula I]\label{metareaI}
Let $f:E\lra Y$ be continuous and assume that the pull-back 
$f^*\nu$ is finite on bounded sets and absolutely continuous with respect to $\mu$.
Then $Jf$ is $\mu$-a.e. finite and for all $\mu$-measurable sets $A\subset E$, we have 
\begin{equation}\label{areametricI}
\int_A Jf(x) \, d\mu(x)=\int_Y N(f,A,y)\, d\nu(y)\,.
\end{equation}
\end{The}
\begin{proof}
Under our assumptions, Theorem~2.9.7 of \cite{Fed} shows that any
$\mu$-measurable set $A\subset X$ is also $f^*\nu$-measurable and the integral formula
\[
f^*\nu(A)=\int_A {\bf D}(f^*\nu,\mu,V,x) \,d\mu(x)
\]
holds, where ${\bf D}(f^*\nu,\mu,V,x)$ is the density of $f^*\nu$ with respect to $\mu$
and the Vitali relation $V$, see 2.9.1 of \cite{Fed}.
By definition of metric Jacobian, for any $\mu$-measurable set $A\subset E$, 
we have $f^*\nu(A)=\int_A Jf(x)\;d\mu(x).$
Thus, formula \eqref{extends} concludes the proof.
\end{proof}
It should be apparent how in the previous theorem the regularity requirements on the
mapping $f$ are transfered to the pull-back measure $f^*\nu$.
These conditions on $f^*\nu$ are satisfied in all known contexts concerning the area
formula and represent the minimal regularity assumptions. 
For instance, they are clearly satisfied for mappings between stratified groups and
then also between Euclidean spaces, \cite{MagPhD}.

Another known metric context is that of Lipschitz mappings from subsets of $\R^n$ 
to metric spaces, equipped with $n$-dimensional Hausdorff measures. 
Here an area formula for Lipschitz mappings from Euclidean spaces to metric spaces 
has been established with different notions of metric Jacobian, \cite{Kir,AmbKir}.
In this framework an a.e. metric differentiability theorem is established 
for Lipschitz mappings and the metric Jacobians are clearly related to the so-called
{\em metric differential}.

In the following example, we wish to present a special context where no reasonable a.e. 
metric differentiability theorem holds. Nevertheless, our metric area formula \eqref{areametricI} holds, 
without referring to any differentiable structure.
\begin{Exa}\rm
Let us consider the identity $I:(\H^1,d)\lra (\H^1,\rho)$ of the Heisenberg group,
that has been constructed in \cite{KirMag2}. Here $d$ is a homogeneous distance of $\H^1$ and 
$\rho$ is a left invariant distance of $\H^1$ that is not homogeneous.
In the above mentioned work, it is proved that $I$ is 1-Lipschitz and {\em nowhere 
metrically differentiable}, according to the notion of \cite{Kir} extended
to the group setting. We have the maximal oscillations 
\begin{equation}\label{oscil}
\limsup_{t\to0^+}\frac{\rho\big(I(x\delta_tz),I(x)\big)}{d(x\delta_tz,x)}=1\quad
\mbox{and} \quad \liminf_{t\to0^+}\frac{\rho\big(I(x\delta_tz),I(x)\big)}{d(x\delta_tz,x)}=0\,.
\end{equation}
Let us equip $(\H^1,d)$ and $(\H^1,\rho)$ with the Hausdorff measure $\cH^4_d$ and
$\cH^4_\rho$, respectively. Since $\cH^4_d$ is doubling on $(\H^1,d)$, by Theorem~2.8.17 of
\cite{Fed}, the {\em covering relation} of closed balls $\{(x,D_{x,r}):x\in \H^1, r>0\}$ 
form an $\cH^4_d$ Vitali relation in $(\H^1,d)$. Furthermore, the injectivity of $I$
gives $f^*\cH^4_\rho(A)=\cH^4_\rho(A)\leq\cH^4_d(A)$ for any $\cH^4_d$-measurable set $A\subset\H^1$.
Clearly $f^*\cH^4_\rho$ satisfies the assumptions of Theorem~\ref{metareaI}, hence we have
\begin{equation}\label{areanodif}
\cH^4_\rho(A)=\int_A JI(x)\,\cH^4_d(x)
\end{equation}
where for all $x\in\H^1$, we have 
$$
JI(x)=J_I(x)=\limsup_{r\to0^+}\frac{\cH^4_\rho(D_{x,r})}{\cH^4_dd(D_{x,r})}=
\limsup_{r\to0^+}\frac{\cH^4_\rho(D_{0,r})}{\cH^4_dd(D_{0,r})}=c_0<+\infty\,.
$$ 
Then we have obtained $\cH^4_\rho=c_0\,\cH^4_d$ with $c_0\geq0$. If we knew that $\cH^4_\rho$
is positive on open sets, then the previous equality would also follow by uniqueness of the
Haar measure in a locally compact Lie group. This positivity of $\cH^4_\rho$ does
not seem a straightforward computation due to the strong oscillations of $\rho$
with respect to $d$, according to \eqref{oscil}. Notice that \eqref{areanodif}
does not refer to any notion of differentiability, although
it turns out to be simple a change of variable formula formula for two different measures.
\end{Exa}

The next lemma is a simple variant of Lemma 2.9.3 in \cite{Fed},
where we replace the Borel regularity of the measure $\zeta$ with the
absolute continuity with respect to $\mu$.
\begin{Lem}\label{differentationlemma}
Let $\zeta$ and $\mu$ be measures that are finite on bounded sets of $X$,
where $\zeta$ is absolutely continuous with respect to $\mu$.
Then for any $\alpha>0$ and any $\mu$-measurable set
$A\subset\big\{x\in X\,\Big|\, (V)\liminf_{S\ra x}\frac{\zeta(S)}{\mu(S)}<\alpha \big\}$,
we have $\zeta(A)\leq \alpha\,\mu(A)$.
\end{Lem}
The next version of the metric area formula uses the more manageable notion of metric Jacobian $J_f$,
hence it requires some additional assumptions on $f$.
Since often one can compare this metric notion of Jacobian with the one related to the differential, 
this theorem can be thought of as a unified approach to the area formula.
%
%
%
%
%
%
%
\begin{The}[Area formula II]\label{metarea}
Let $f:E\lra Y$ be continuous and assume that the pull-back $f^*\nu$ is finite on bounded
sets and absolutely continuous with respect to $\mu$.
If $A\subset E$ is $\mu$-measurable and there exist 
disjoint $\mu$-measurable sets $\{E_i\}_{i\in\N}$ such that
\[
\mu\Big(E\sm\bigcup_{i\in\N}E_i\Big)=0,
\]
$f_{|E_i}$ is injective for every $i\ge1$ and $J_f(x)=0$ for $\mu$-a.e. $x\in E_0$,
then we have
\begin{equation}\label{areametric}
\int_A J_f(x) \, d\mu(x)=\int_Y N(f,A,y)\, d\nu(y)\,.
\end{equation}
\end{The}
\begin{proof}
We can assume that any $E_i$ is contained in $E$. 
Let us fix $\ep>0$ and consider a sequence of closed sets
$C_i\subset E_i$ such that $\mu(E_i\sm C_i)\leq\ep 2^{-i}$ for any $i\in\N$.
Let us set $f_i=f_{|C_i}$ and notice that for all $x\in C_i$ we have
\[
J_{f_i}(x)=(V)\limsup_{S\to x}\frac{\nu\big(f(S\cap C_i)\big)}{\mu(S)}\leq(V)\limsup_{S\to x}
\frac{\nu\big(f(S\cap E)\big)}{\mu(S)}=J_f(x).
\]
By Corollary~2.9.9 of \cite{Fed} applied to both ${\bf 1}_{C_i}$ and ${\bf 1}_{C_i}\,Jf$, 
it follows that for $\mu$-a.e. $x\in C_i$, we have
\begin{eqnarray}
&& (V) \lim_{S\to x} \frac{1}{\mu(S)}\int_{S} {\bf 1}_{C_i}(z)\,{\bf D}(f^*\nu,\mu,V,z)\,d\mu(z)=Jf(x)\,, \label{VlimL1} \\
&& (V) \lim_{S\to x} \frac{1}{\mu(S)}\int_{S} {\bf D}(f^*\nu,\mu,V,z)\,d\mu(z)=Jf(x)\,. \label{VlimL2}
\end{eqnarray}
Now, for all $x\in C_i$ such that \eqref{VlimL1} and \eqref{VlimL2} hold, we have
\begin{eqnarray*}
J_f(x)&=&(V)\limsup_{S\to x}\frac{\nu\big(f(S\cap E)\big)}{\mu(S)} \\
&\leq&(V)\limsup_{S\to x}\frac{f^*\nu(S\cap E)}{\mu(S)}=Jf(x) \\
&\leq&(V)\limsup_{S\to x}\frac{\nu\big(f(S\cap E\cap C_i)\big)}{\mu(S)}
+(V)\limsup_{S\to x}\frac{f^*\nu(S\cap E\sm C_i)}{\mu(S)} \\
&\leq&(V)\limsup_{S\to x}\frac{\nu\big(f_i(S\cap C_i)\big)}{\mu(S)}
+(V)\limsup_{S\to x}\frac{f^*\nu(S\sm C_i)}{\mu(S)} \\
&=&(V)\limsup_{S\to x}\frac{\nu\big(f_i(S\cap C_i)\big)}{\mu(S)}\,.
\end{eqnarray*}
The last equality follows by both \eqref{VlimL1} and \eqref{VlimL2}, hence we get
$J_f(x)=Jf(x)=J_{f_i}(x)$. These equalities hold a.e. in $C_i$ for any $i\geq1$.
Let $B_1=\cup_{i=1}^\infty C_i$ and let $A_1=\cup_{i=1}^\infty E_i$.
Then we have $\mu(A_1\sm B_1)\leq \ep$, where we have shown that
the previous equalities of metric Jacobians hold $\mu$-a.e. in $B_1$. 
The arbitrary choice of $\ep$ allows for constructing an increasing sequence of
Borel sets $B_i\subset A_1$ such that $\mu(A_1\sm B_n)\leq \ep/n$ for all $n\geq1$.
In particular, setting $B_\infty=\cup_{n=1}^\infty B_n$, we have that
\[
\mu(A_1\sm B_n)\searrow\mu(A_1\sm B_\infty)
\]
as $n\to\infty$ and this limit is zero. 
Thus, in view of formula \eqref{areametricI}, we get
\begin{eqnarray*}
f^*\nu(A\cap A_1)=\int_{A\cap A_1}Jf(x)\,d\mu(x)=\int_{A\cap A_1}J_f(x)\,d\mu(x)\,. 
\end{eqnarray*}
We have obtained the formula
\begin{eqnarray}\label{prearea}
f^*\nu(A)&=&\int_{A\cap A_1}J_f(x)\,d\mu(x)+f^*\nu(A\cap E_0)\,.
\end{eqnarray}
We have to show that $f^*\nu(A\cap E_0)=0$.
Let us consider for any $Z\subset X$ the ``preimage measure'' $f^\sharp\nu(Z)=\nu(f(Z))$
that is absolutely continuous with respect to $\mu$.
Since the set where $J_f>0$ in $E_0$ is $\mu$-negligible and $f^*\nu$ is absolutely
continuous with respect to $\mu$, it is not restrictive to assume that 
$J_f$ everywhere vanishes on $E_0$. Now, for every $\epsilon>0$ and every $\mu$-measurable
bounded set $F\subset E_0$, we get $f^\sharp\nu(F)\leq\epsilon\mu(F)$, due to
Lemma~\ref{differentationlemma} applied with $\zeta=f^\sharp\nu$. 
This clearly implies $f^\sharp\nu(E_0)=\nu\big(f(E_0)\big)=0$, hence \eqref{extends} gives $f^*\nu(E_0)=0$.
Then \eqref{prearea} easily guides us to the conclusion.
\end{proof}
%
%
%
%
%
%
%
%
%
%
%
%
\begin{Rem}\rm
The metric area formulae \eqref{areametricI} and \eqref{areametric} 
can be extended to all nonnegative measurable mappings $u:A\lra[0,+\infty]$,
obtaining 
\begin{equation}
\int_A\,u(x)\,J_f(x)\,\,d\mu(x)= \int_Y\sum_{x\in
f^{-1}(y)}\!\!\!\!\!u(x)\,\,d\nu(y)\,.
\end{equation}
This follows by standard approximation arguments with measurable step functions.
\end{Rem}
%
%
%
%
%
\begin{Rem}\rm
Let $(X,d)$ be a complete and separable metric space, let $\alpha>0$,
let $(Y,\rho)$ be a metric space
and consider the metric measure spaces $(X,d,\cH^\alpha_d)$ and $(Y,\rho,\cH^\alpha_\rho)$.
Let $E\subset X$ be closed and let $f:E\lra Y$ be a Lipschitz mapping.
We assume that
\begin{enumerate}
\item
$\cH^\alpha_d$ is finite on bounded sets of $X$,
\item
for $\cH_d^\alpha$-a.e. $x\in X$ the inequality
$\displaystyle \liminf_{r\to0^+}\frac{\cH^\alpha(D_{x,r})}{r^\alpha}>0$ holds.
\end{enumerate}
These conditions easily imply that $C=\{(x,D_{x,r}): x\in X, \, r>0\}$ 
is $\cH^\alpha_d$ Vitali relation, hence the areae formulae of Theorem~\ref{metareaI} and 
Theorem~\ref{metarea}
hold for $f$, where $\mu=\cH^\alpha_d$ and $\nu=\cH^\alpha_\rho$. It follows that 
the metric area formula of \cite{MagPhD} follows as a special case of our Theorem~\ref{metarea}.
\end{Rem}
\begin{Rem}\rm
Let $E$ be a closed subset of $\R^n$, let $(Y,\rho)$ be a metric space equipped with the 
masure $\cH^n_\rho$ and let $f:E\lra Y$ be Lipschitz. The conditions of the previous remark are clearly
satisfied with the Lebesgue measure $\cL^n$ on $\R^n$. 
Then our area formulae hold, along with that of \cite{AmbKir} and \cite{Kir}.
In particular, all the different notions of metric Jacobian that are 
involved in these formulae coincide $\cL^n$-a.e.
\end{Rem}

\end{document}